\documentclass[draftclsnofoot,onecolumn]{IEEEtran}
\usepackage{cite}
\ifCLASSINFOpdf

\else

\fi

\usepackage{amsbsy}
\usepackage{amsmath}
\usepackage{amssymb}
\usepackage{graphicx}
\usepackage{amsfonts,amstext,mathrsfs}
\usepackage{amsthm}

\newtheorem{theorem}{Theorem}
\newtheorem{lemma}{Lemma}
\newtheorem{corollary}{Corollary}

\newtheorem{definition}{Definition}

\newcommand{\mI}{\hbox{{\bf I}}}

\newcommand{\gth}{\theta}

\newtheorem{question}{Question}[section]
\newtheorem{coro}{Corollary}[section]

\newcommand{\beq}{\begin{equation}}
\newcommand{\eeq}{\end{equation}}
\newcommand{\bea}{\begin{array}}
\newcommand{\ena}{\end{array}}
\newcommand{\bds}{\begin {itemize}}
\newcommand{\eds}{\end {itemize}}
\newcommand{\bdf}{\begin{definition}}
\newcommand{\blm}{\begin{lemma}}
\newcommand{\edf}{\end{definition}}
\newcommand{\elm}{\end{lemma}}
\newcommand{\bthm}{\begin{theorem}}
\newcommand{\ethm}{\end{theorem}}
\newcommand{\bprp}{\begin{prop}}
\newcommand{\eprp}{\end{prop}}
\newcommand{\bcl}{\begin{claim}}
\newcommand{\ecl}{\end{claim}}
\newcommand{\bcr}{\begin{coro}}
\newcommand{\ecr}{\end{coro}}
\newcommand{\bquest}{\begin{question}}
\newcommand{\equest}{\end{question}}

\begin{document}

\title{On the Non-Existence of Unbiased Estimators in Constrained Estimation Problems}
\author{Anelia Somekh-Baruch, Amir Leshem and Venkatesh Saligrama\thanks{A.\ Somekh-Baruch and Amir Leshem are with the Faculty of Engineering at Bar-Ilan University, Ramat-Gan, Israel.  Emails: somekha@biu.ac.il, leshema@biu.ac.il.  Amir Leshem's work was funded by the visiting scholar program of Boston University. Venkatesh Saligrama is with the Department of Electrical and Computer Engineering, Boston University, Email: srv@bu.edu. This material is based upon work supported in part by NSF Grants CCF: 1320566, NSF Grant CNS: 1330008 NSF CCF: 1527618,  the U.S. Department of Homeland Security, Science and Technology Directorate, Office of University Programs, under Grant Award 2013-ST-061-ED0001, by ONR Grant 50202168 and US AF contract FA8650-14-C-1728. The views and conclusions contained in this document are those of the authors and should not be interpreted as necessarily representing the social policies, either expressed or implied, of the NSF, U.S. DHS, ONR or AF.}}

\maketitle

\begin{abstract}
We address the problem of existence of unbiased constrained parameter estimators. We show that if the constrained set of parameters is compact and the hypothesized distributions are absolutely continuous with respect to one another, then there exists no unbiased estimator. Weaker conditions for the absence of unbiased constrained estimators are also specified.
We provide several examples which demonstrate the utility of these conditions. 
\end{abstract}

\begin{IEEEkeywords}
Cramer-Rao bound, unbiased estimation, Estimation theory, constrained estimators
\end{IEEEkeywords}

\vspace{7cm}

\pagebreak

\section{Introduction}

Unbiased estimation is a cornerstone of classical estimation theory and arises as an important concept in minimum variance unbiased estimation (MVUE) theory \cite{LehmannCasella98}, \cite{VanTrees2004detection}, \cite{scharf1991statistical}. Information theoretic bounds are an important tool for the evaluation of estimators for identifying theoretical performance gaps.  \cite{barankin1949locally}, \cite{bell1997extended}, \cite{ziv1969some},  \cite{weinstein1988general}. 

In this context, Cramer-Rao bound (CRB) has emerged as an important information inequality since it not only serves as a non-trivial lower bound on all unbiased estimators but is also often readily computatable particularly for unconstrained parameter estimation problems. While CRB can also be extended to bound the performance of biased estimators, these bounds are often not useful in many interesting cases. This is because they require the knowledge of the bias as a function of the parameters and are thus not agnostic to estimation scheme employed. 

In many real-world estimation problems we often encounter constraints on the parameter space in the form of side-information.  
For example in many communication systems we encounter positivity constraints, limited power constraints, bandwidth or delay constraints, circularity constraints, subspace constraints, and so on. Consequently, it makes sense to incorporate these constraints for deriving constraint dependent lower bounds.  Motivated by these scenarios a significant amount of research has been conducted on developing estimation techniques and performance bounds for constrained problems. In this context, many papers propose deriving constrained CRBs and in particular focus on bounds for unbiased estimators on account of its relative computational simplicity. This raises the question of as to the conditions under which unbiased estimators can exist in constrained scenarios. Before we describe our results we will discuss some of the related work in this context below.

CRBs for constrained parameters was introduced and studied by Gorman and
Hero \cite{gorman1990lower}. A constrained CRB on the error covariance of estimators of multi-dimensional parameters was derived. The derivation was based on a limiting form of a multi-parameter Chapman-Robbins \cite{chapman1951minimum} form of a Barankin-type bound \cite{barankin1949locally}.
It was shown that the constrained CRB is equivalent to the unconstrained CRB evaluated with a ``constrained Fisher information matrix" for the case in which there is a general smooth functional inequality constraint of the form ${\cal G}_{\theta}\leq 0$. This constrained Fisher matrix was shown to be identical to the classical
unconstrained Fisher matrix at all regular points of the constraint set, e.g., at interior points. However as noted by \cite{gorman1990lower} at non-regular points, such as points governed by equality constraints,
the constrained Fisher matrix was observed to be a rank-deficient matrix.
It was also established that functional constraints necessarily decrease the CR
bound for unbiased estimators.

Marzetta \cite{marzetta1993simple} provided a simpler proof for the case of constrained {\it unbiased} estimators.
While Gorman and Hero's derivation relied on an application of
the Cauchy-Schwarz inequality to a pair of random vectors, one of
which has a possibly singular covariance matrix, the
derivation of Marzetta avoids using the Cauchy-Schwarz inequality and it avoids pseudoinverses. But, it focuses solely on unbiased estimators. Another result of \cite{marzetta1993simple}
is a necessary condition for an unbiased estimator to achieve the
constrained CRB with equality.

A subsequent work by Stoica and Ng \cite{StoicaBoon1998} extends the constrained CRB (under differentiable,
deterministic constraints on the parameters) to the case in which the Fisher information matrix for
the unconstrained problem is not necessarily full rank. This case was not treated in previous works. 
It was assumed in \cite{StoicaBoon1998} that the estimator was unbiased and it was shown that the expression for the constrained CRB in this case depends
only on the unconstrained Fisher information matrix and a basis of the nullspace of the
constraints gradient matrix.
A necessary and sufficient
condition for the existence of the constrained CRB was also derived. 

There have been a number of follow-on works that have utilized constrained CRBs for a range of applications including sparse estimation problems~\cite{haim_eldar} and MIMO Radars~\cite{Wang2016}. 

Unfortunately, as we shall see in this paper, in many interesting cases of constrained parameter sets, an unbiased estimator does not exist. This limits the applicability of the results, mainly to cases of biased estimators, and to cases of unbiased estimators in which the constrained parameter set is a manifold without boundaries, i.e., there are no non-trivial inequality constraints.
Moreover, often, the use of the CRBs for biased estimators is quite complicated
since it requires explicit expressions for both the bias
and the gradient of the bias. 
\emph{One of the most striking consequences of the results of this paper is that if the constraint set is the set of solutions (zeros) of a set of continuous non-linear functions which is bounded, then the exists no unbiased estimator. }

This result in turn provides an explanation for the observation of \cite{gorman1990lower} that inequality constraints do not reduce the bound at internal points of the constrained parameter set. The variance is not reduced simply because there are no constrained unbiased estimators in this case. This implies that in order to use the Gorman-Hero bound, one needs to know the bias of the estimator with the constraint.

In this paper we focus on the case of an estimation problem of vector of parameter $\theta$ that lies in $\mathbb{R}^d$. 
Our first result concerns the case where $\theta$ is constrained to a set, which has at least one extreme point, $\theta_e$. In addition we assume that the distribution of the observed random variable $X$ given $\theta_e$ is absolutely continuous with respect to (w.r.t.) the distribution of $X$ given another parameter value in the constrained set. Under these assumptions we show that there exists no unbiased estimator whose range is the constrained set. 
We then specialize the result to the important case of a constrained parameter set which is compact (and not a singleton) and show that the result continues to hold in this special case. 
Furthermore, we extend the result to a case of a bounded constraint set, which can be open, and whose closure contains an extreme point and where the distribution of the observed random variable has a bounded Radon Nikodym derivative. 
We further provide examples to demonstrate that the conditions we specify are necessary and cannot be further relaxed. 

Our results imply that almost in every constrained problem that one can think of, there exists no unbiased estimator. 
This result is surprising in light of the scarcity of examples which appear in the literature for the non existence of unbiased constrained estimators (e.g.\ \cite{RomanoSiegel86}). 
In fact, the non-existence of unbiased estimators is the more common/natural case, a fact that will be defined more explicitly. 
Moreover, this fact has an important implication on CR bounds in the constrained case; 
 since the applicable case is the biased one, the CR bound which requires explicit expressions for both the bias
and the gradient of the bias might be useless in many setups. 

\section{Main Results}

We begin by presenting a theorem which specifies conditions under which an unbiased estimator of a vector of parameters does not exist.
\begin{lemma}\label{th: closed_set theorem}
 Let $  \Theta\subseteq\mathbb{R}^d$ be a set of parameters. 
 Let $\Theta' \subseteq \Theta$ be a subset, which is not a singleton, and let $P_{\theta},\; \theta\in \Theta' $ be a corresponding collection of distributions. 
 Let ${\mathbf{X}}\sim P_{\theta}$ be a random variable over some measurable space.
Suppose there exists an extreme point $\theta_e\in \Theta'$ and an additional point $\theta'\in \Theta'$, $\theta' \neq \theta$, such that $P_{\theta'}$ and $P_{\theta_e}$ are absolutely continuous w.r.t.\ one another.    
Then there exists no unbiased estimator $g(\mathbf{X})$ for $\theta$, which is measurable w.r.t.\ Lebesgue $\sigma$-algebra whose range is $\Theta'$.
\end{lemma}
Note that
Theorem \ref{th: closed_set theorem} as well as the proceeding Theorem \ref{cr: closed_set theorem} do not pose restrictive assumptions on the measurable space on which the random variable $\mathbf{X}$ is defined. Hence, as a special case, it holds for $\mathbf{X}$ which is a random $N$-Vector over $\mathbb{R}^N$. 

\begin{proof}
Let $\theta_e$ be an extreme point of $\Theta'$, the existence of which is guaranteed by assumption. 

Assume that $g(\mathbf{X})$ is an unbiased estimator of $\theta$ whose range is $\Theta'$.
For notational convenience we denote by $P(g(\mathbf{X})|\theta)$ the distribution of $g(\mathbf{X})$. 
By unbiasedness we have
\beq\label{eq: unbiasedness definition}
\mathbb{E}_{P_{\theta_e}} \left(g(\mathbf{X}) |\theta_e\right) = \theta_e,
\eeq
where $\mathbb{E}_{P}(\cdot) $ denote expectation w.r.t.\ the distribution $P$. 
Therefore, by definition of the expectation we obtain
\begin{flalign}\label{eq: conv comb}
\theta_e=& \mathbb{E}_{P_{\theta_e}} \left(g(\mathbf{X})|\theta_e\right)\nonumber\\
=&\int_{\cal X} g(\mathbf{X})dP_{\theta_e} \nonumber\\
=&\int_{\Theta'} g(\mathbf{X})dP(g(\mathbf{X})|\theta_e),
\end{flalign}
where ${\cal X}$ is the alphabet of $\mathbf{X}$ and the last step follows from the assumption that the range of $g$ is $\Theta'$.

While the RHS \ of (\ref{eq: conv comb}) is a convex combination of points in $\Theta'$, 
the LHS \ of (\ref{eq: conv comb}), that is, $\theta_e$, is an extreme point of $\Theta'$.
By definition of an extreme point in $\Theta'$ as a point which does not lie in any open line segment joining two points of $\Theta'$, we conclude that
\begin{flalign}\label{eq: last repeated step}
g(\mathbf{X})\equiv \theta_e \mbox{ w.p.\ $1$ w.r.t\ $P_{\theta_e}$}.
\end{flalign}
Since by assumption, the distributions $P_{\theta'}$ and $P_{\theta_e}$ 
are absolutely continuous w.r.t.\ one another 
it follows that\footnote{To realize that (\ref{eq: sfdhiuhu}) holds, assume in negation that it does not. Thus $\exists \epsilon>0$ such that $P_{\theta'}(\{x:\; g(x)\neq \theta_e \})>\epsilon$, while from (\ref{eq: last repeated step}) we have $P_{\theta_e}(\{x:\; g(x)\neq \theta_e \})=0$, contradicting the absolute continuity of $P_{\theta'}$ and $P_{\theta_e}$ w.r.t.\ one another.} 
\begin{flalign}\label{eq: sfdhiuhu}
g(\mathbf{X})\equiv \theta_e \mbox{ w.p.\ $1$ w.r.t.\ $P_{\theta'}$},
\end{flalign}
which contradicts the unbiasedness of the estimator for $\theta'$. 
\end{proof}
The following theorem gives alternative (stricter) conditions for the non existence of an unbiased estimator. 
\begin{theorem}\label{cr: closed_set theorem}
 Let $  \Theta\subseteq\mathbb{R}^d$ be a set of parameters.
 Assume further that $\Theta' \subseteq \Theta$ is a compact set which is not a singleton, and that $P_{\theta},\; \theta\in \Theta' $ is a collection of distributions,  which are absolutely continuous w.r.t.\ one another. 
Let ${\mathbf{X}}\sim P_{\theta}$ be a random variable over some measurable space.

  Then there exists no unbiased estimator $g(\mathbf{X})$ for $\theta$, which is measurable w.r.t.\ Lebesgue $\sigma$-algebra whose range is $\Theta'$.
\end{theorem}
\begin{proof}
The proof of Theorem \ref{cr: closed_set theorem} is a straightforward consequence of Theorem \ref{th: closed_set theorem}, which is established by proving that a compact set in $\mathbb{R}^d$ must contain at least one extreme point. 
To this end, we first state a known result, including its proof for the sake of completeness. 
\begin{lemma}\label{lm: compact lemma}
The convex hull of a compact set in $\mathbb{R}^d$ is compact.
\end{lemma}
\begin{proof}
Let ${\cal P}(d+1)$ be the simplex of probability vectors of length ${d+1}$, i.e.,
\begin{flalign}
{\cal P}(d+1)\triangleq \left\{ \left(\lambda_1,...,\lambda_{d+1}\right):\; \sum_{i=1}^{d+1} \lambda_i ,0\leq \lambda_i\leq 1\right\}.
\end{flalign}
Note that ${\cal P}(d+1)$ is closed and bounded and is therefore compact. 

Recall Carath\'{e}odory's Theorem which states that in a convex set in $\mathbb{R}^d$, every point can be expressed as a convex combination of $d+1$ points. 
Therefore, one can define the convex hull of a set ${\cal A}\in\mathbb{R}^d$ as
\begin{flalign}
&Conv({\cal A})\nonumber\\
\triangleq & \left\{x:\; \exists \begin{array} {ll} y_1,...,y_{d+1}\in {\cal A},\\
\left( \lambda_1,...,\lambda_{d+1}\right)\in {\cal P}(d+1)
 \end{array} 
 , x=\sum_{i-1}^{d+1} \lambda_i y_i\right\}.
\end{flalign}
Let ${\cal A}\subseteq \mathbb{R}^d$ be a compact set. Assume that $x=\lim_{n\rightarrow \infty}x_n$ where $x_n\in Conv({\cal A})$. To prove that $Conv({\cal A})$ is compact, it should be established that $x\in Conv({\cal A})$.

By Carath\'{e}odory's Theorem, each $x_n$ can be written in the form
$x_n=\sum_{i=1}^{d+1} \lambda_{i,n} y_{i,n}$ where $y_{i,n}\in{\cal A}$ and $(\lambda_{1,n},...,\lambda_{d+1,n} )\in{\cal P}(d+1)$.

Now, since ${\cal A}$ and ${\cal P}(d+1)$ are compact, there exists a sequence $n_1,n_2,...$ such that the limits $\lim_{\ell\rightarrow\infty }\lambda_{i,n_\ell}=\lambda_i$ and $\lim_{\ell\rightarrow\infty}y_{i,n_\ell}=y_i$ exist for $i=1,...,d+1$. 
  Clearly $\lambda_i\geq 0 $, $\sum_{i=1}^n \lambda_i=1$ and $y_i\in {\cal A}$. 
Thus, the sequence $x_n, n=1,2,...$ has a subsequence, $x_{n_\ell},\ell=1,2,...$ which converges to a point in $Conv({\cal A})$ and this establishes the fact that $Conv({\cal A})$ is compact.
\end{proof}
Now, we can use Lemma \ref{lm: compact lemma} to obtain the desired result. 
\begin{lemma}
The set of extreme points of a compact set in $\mathbb{R}^d$ is non-empty. 
\end{lemma}
\begin{proof}
By the Krein-Milman Theorem \cite{Krein1940}, the closure of $Conv({\cal A})$ is the convex hull of its extreme points. If ${\cal A}$ is compact, then by the previous lemma, $Conv({\cal A})$ is closed. Therefore, $Conv({\cal A})$ is the convex hull of its extreme points. Hence, all the extreme points of $Conv({\cal A})$ must belong to ${\cal A}$.
By definition, an extreme point is not a convex combination of other points in the set, hence all the extreme points must belong to ${\cal A}$. 
\end{proof}
Now, repeating steps (\ref{eq: unbiasedness definition})-(\ref{eq: last repeated step}) the condition that $P_{\theta},\; \theta\in \Theta' $ are absolutely continuous w.r.t.\ one another, ensures that (\ref{eq: sfdhiuhu}) holds for all $\theta'\neq \theta_e$, which contradicts the unbiasedness of the estimator for all $\theta'\neq \theta_e$.

\end{proof}

We next present a corollary of Theorem \ref{cr: closed_set theorem}.
\begin{corollary}\label{cr: continuous non linear f}
Let $f$ be a continuous non-linear function and let the set $\Theta'=\left\{\theta:f(\theta)=0\right\}$ of solutions be bounded, then there is no unbiased estimator. The same result holds for a bounded set $\Theta'=\left\{\theta:f(\theta)\leq 0\right\}$.  
\end{corollary}
\begin{proof}
By the boundedness and the continuity of $f$ we have that the set $\Theta'$ must also be closed and thus also compact. therefore, the conditions of Theorem \ref{cr: closed_set theorem} are met and there exists no unbiased estimator. 
\end{proof}

Next, we relax the condition that there must exist an extreme point within the constraint set, and replace it with a requirement that an extreme point exists on its closure. This allows to generalize the results to open constraint sets under an additional requirement on set of distributions $P_{\theta},\; \theta\in \Theta' $. We denote by $\overline{{\cal A}}$ the closure of the set ${\cal A}$. 

\begin{theorem}\label{th: DCT theorem}
 Let $  \Theta\subseteq\mathbb{R}^d$ be a set of parameters. 
 Let $\Theta' \subseteq \Theta$ be a subset (not a singleton) whose closure is bounded. 
Let $P_{\theta},\; \theta\in \overline{\Theta'}$ be a collection of distributions which are absolutely continuous w.r.t.\ the Lesbegue measure with Radon Nikodym derivatives (p.d.f.'s) $f_{\theta},\; \theta\in \overline{\Theta'}$ which are uniformly bounded. 
 Let ${\mathbf{X}}\sim P_{\theta}$ be a random variable over some measurable space.
Suppose there exists an extreme point $\theta_e\in \overline{\Theta'}$ and that there exists a sequence $\theta_n$ which satisfies $\lim_{n\rightarrow \infty}\theta_n=\theta_e$ and the sequence $f_{\theta_n}$ converges globally in Lesbegue measure\footnote{Global convergence in Lesbegue measure of the sequence $f_{\theta_n}$ to $f_{\theta_e} $ means that for every $\epsilon>0$,  $\lim_{n\rightarrow\infty} \mu(\{x:| f_{\theta_n}(x)-f_{\theta_e} (x)|>\epsilon\})=0$, $\mu$ being the Lesbegue measure.} to the p.d.f.\ $f_{\theta_e}$.
Then there exists no unbiased estimator $g(\mathbf{X})$ for $\theta$, which is measurable w.r.t.\ the Lebesgue $\sigma$-algebra whose range is $\Theta'$.
\end{theorem}
\begin{proof} 
Let $g(\mathbf{X})$ be a given estimator.
Consider the sequence of measurable mappings $A_n=g(\mathbf{X})\cdot f_{\theta_n}(\mathbf{X})$, n=1,2,... 
By the assumptions and by definition, $A_1,A_2,...$ is a uniformly integrable sequence which converges globally in Lesbegue measure to the limit $A=g(\mathbf{X})\cdot f_{\theta_e}(\mathbf{X})$.  
Hence, there exists a subsequence $A_{n_k}, k=1,2,..$ which converges a.s.\ to $A$. 

Since $f_{\theta_n}(\mathbf{X})$ is uniformly bounded by say $M<\infty$, and since $g(\mathbf{X})$ must lie in $\Theta'$ which is also assumed to be bounded,
one can invoke the Lesbegue dominated convergence Theorem yielding
\begin{flalign}\label{eq: DCT}
\lim_{k\rightarrow\infty} \mathbb{E}_{P_{\theta_{n_k}}}(g(\mathbf{X}))= \mathbb{E}_{P_{\theta_e}}(g(\mathbf{X})).
\end{flalign}

On the other hand, from the unbiasedness we have $\mathbb{E}_{P_{\theta_{n_k}}}(g(\mathbf{X}))=\theta_{n_k}$ and therefore, 
\begin{flalign}\label{eq: DCT 2}
\lim_{k\rightarrow\infty} \mathbb{E}_{P_{\theta_{n_k}}}(g(\mathbf{X}))= \lim_{k\rightarrow\infty}\theta_{n_k}=\theta_e,
\end{flalign}
where the right inequality follows by assumption. 

Equations (\ref{eq: DCT}) and (\ref{eq: DCT 2}) imply that
\begin{flalign}\label{eq: DCT 3}
\mathbb{E}_{P_{\theta_e}}(g(\mathbf{X}))=\theta_e,
\end{flalign}
which means that the estimator is unbiased also at the point $\theta_e$. Therefore, the conditions of Lemma \ref{th: closed_set theorem} are satisfied with $\Theta'$ substituted by $\Theta'\cup \theta_e$ and the theorem follows. 

\end{proof}
We note that the theorem can be extended in several manners, for example when the sequence of measures $P_{\theta_n}$ is tight, the convergence in measure follows from Prokhorov's Theorem \cite{Prokhorov56}. 

\section{Examples and Discussion}

In this section we present some examples of unbiased constrained estimators, and characterize some important cases in which unbiased estimators do not exist. 

\subsection{Example: IID Bernoulli Random Variables}
The classical example for unbiased (and even MVU) estimator over a compact parameter space is the estimation of the parameter of an IID sequence of Bernoulli random variables with parameter $0 \le p \le 1$.
In this case, the sample mean is the best unbiased estimator. However, at the maximal and minimal values of the parameter, the distribution of the measurements is concentrated at a single value, i.e.,
when $p\in \left\{0,1\right\}$ all the measurements assume the value $p$ exactly. 
This is not a coincidence as we saw above; had the distribution of the measurements with $p=0$ been absolutely continuous w.r.t.\ any other distribution with $p\in(0,1)$, an unbiased estimator could not have existed. 

This example can now be generalized to obtain non-trivial constraints: Assume that $\theta\in [0 , 2]$ and that given $\theta$, $X$ is a random vector with value $\lfloor \gth \rfloor+z$ where $z$ is a Bernoulli random IID vector with mean $\gth-\lfloor \gth \rfloor$. Similar to the example above we can find an unbiased estimator of $\gth$ when we require $1 \le \gth \le 2$.

Moreover, in \cite[Chapter 7.12]{RomanoSiegel86}, the problem of estimating the odds ratio $\frac{p}{1-p}$ in $N$ independent Bernoulli trials with probability of success $p$ is considered, as an example of non-existence of unbiased estimator. Any statistics $T$ which maps the observable binary string of length $N$ to a real number, would result in an expectation $\mathbb{E}(T)=\sum_{j=1}^{2^n}t_j p^{n_j}(1-p)^{n-n_j}$. Thus, an unbiased estimator should satisfy the condition that this polynomial would be equal to $\frac{p}{1-p}$ for all $p\in(0,1)$, which is clearly impossible. 

\subsection{Example: Estimation of the Variance of IID Gaussian Random Variables}
A second interesting example is the case of estimating the variance of a Gaussian random vector from a sequence of $N$ IID measurements $X \sim{\cal N}(0,\gth \mI)$. An unbiased estimator exists with the constraint
$\gth \ge 0$. However, as a consequence of our results, one cannot find an unbiased estimator with any other constraint of the form $\gth \ge c$, where $c$ is strictly positive.
Since the point $\gth_e = c$ is an extreme point of the interval $[c,\infty)$, and for any other $\theta>c$, $P_{\theta}$ is absolutely continuous w.r.t.\ $P_{\theta_e}$ (as both are absolutely continuous w.r.t.\ the Lebesgue measure on $\mathbb{R}^N$), there exists no unbiased estimator of $\theta$, which is confined to $[c,\infty)$. On the other hand, for $\gth \ge 0$ we obtain that although $\theta=0$ is an extreme point, since a zero variance random variable is deterministic, its distribution is not absolutely continuous w.r.t.\ any other positive-variance Gaussian distribution, and therefore the conditions of Theorem \ref{th: closed_set theorem} do not hold and indeed an unbiased estimator for $\theta$ which lies in $[0,\infty)$ exists.

Further, if one considers the open constraint set $\Theta'=(c_1,c_2)$, the conditions of Theorem \ref{th: DCT theorem} are satisfies, and again there exists no unbiased estimator. 

\subsection{Example: Spectrum Estimation with Power Constraints:}
We next revisit Example 3 discussed in \cite{gorman1990lower}: 
Let $(X_1,....,X_N)$ be a segment of a real wide sense stationary
random process with power spectral density (PSD) $S_X(f), f\in[-1/2,1/2]$. 
The objective is to estimate the PSD,
at $d$ distinct frequencies  $\theta_i= S_X(f_i)$, $f_1,...,f_d$. It is assumed that the
average power of $\{X_i\}$ is known over $p$ non-overlapping
frequency bands, i.e., the vector $\theta^d=(\theta_1,...,\theta_d)$
is constrained to satisfy the equations
\begin{flalign}
\left[\begin{matrix} \chi_{1,1} & \chi_{1,2} & \ldots & \chi_{1,d}\\
 \chi_{2,1} & \chi_{2,2} & \ldots & \chi_{2,d}\\
\vdots & \vdots & \ddots & \vdots\\
   \chi_{p,1} & \chi_{p,2} & \ldots & \chi_{p,d}
   \end{matrix} \right]    
   \left[\begin{matrix} \theta_1\\\theta_2\\ \vdots\\ \theta_d\end{matrix} \right]=&\left[\begin{matrix} \mathbb{E}_1\\\mathbb{E}_2\\ \vdots\\ \mathbb{E}_p\end{matrix} \right],
\end{flalign}
where $\{\chi_{i,j}\}, i=1,..,p, j=1,...,d$ is a specified matrix with binary $(0,1)$ entries, and $\mathbb{E}_i, i=1,...,p$ are specified power levels. 
As mentioned in \cite{gorman1990lower}, these are in fact 
$p$ linear constraints
on the unknown PSD, known as the $p$-point
constraint in robust Wiener filtering theory. 
In this example, the constraint set has no extreme points and therefore an unbiased estimator can exist.

\subsection{Example: Intersection of Polyhedral Sets}

Consider the case of inequality constraints
\begin{flalign}\label{eq: intersection polyhedrals}
\sum_{i=1}^d \alpha_{i,k} \theta_i\leq B_k, \; k=1,...,d, 
\end{flalign}
where $B_k, \alpha_{i,k}, i\in \{1,...,d\}, k\in \{1,...,d\}$ are real constants. In the case in which the resulting set of allowable $\theta$'s in (\ref{eq: intersection polyhedrals}) defines a polyhedron, or even a set which possesses an extreme point, there exists no unbiased estimator for $(\theta_1,...,\theta_d)$. 

\subsection{Example: Continuous Non-linear Function with Equality Constraints}

Revisiting the case discussed in Corollary \ref{cr: continuous non linear f}, of a bounded set $\Theta'=\left\{\theta:f(\theta)=0\right\}$ or $\Theta'\leq \left\{\theta:f(\theta)=0\right\}$ of solutions; one such example is a $p$-norm equality constraint, i.e., $\|\theta\|_p= R$. 
Other examples are when the parameter satisfies continuous non-linear inequality constraints, e.g., 
$||\theta|| \leq R$ 
or when $\theta$ belongs to an ellipsoid in $\mathbb{R}^d$. In these cases too the same argument holds since the boundary is the equality set.

\subsection{Example: Constrained Sparse Estimation Problem}

Consider an observed vector 
\begin{flalign}
\mathbf{X}=& A\theta+\mathbf{Z}
\end{flalign}
where $\mathbf{Z}$ is a Gaussian $n$-vector, $A$ is a known deterministic $n\times d$ matrix, and $\theta=(\theta_1,...,\theta_d)^T$ is a sparse $d$-vector of unknown parameters which satisfies
\begin{flalign}
\sum_{i=1}^d|\theta_i|\leq 1.
\end{flalign} 
Since the conditions of Theorem \ref{cr: closed_set theorem} are satisfied, there exists no unbiased estimator $g(\mathbf{X})=(g_1(\mathbf{X}),...,g_d(\mathbf{X}))$ for $\theta$ which satisfies $\sum_{i=1}^d| g_i(\mathbf{X}) |\leq 1$. 
\subsection{Extension to the case of periodic unbiasedness}
The papers \cite{RouttenbergTabrikian2011}, \cite{routtenberg2013non}, 
\cite{routtenberg2012performance}, \cite{routtenberg2014cramer} consider mean square periodic error criterion combined with periodic unbiasedness for which the conventional CR bound does not provide a valid bound. 
Lehmann-unbiasedness concept\footnote{An estimator $g(X)$ is said to be Lehmann unbiased w.r.t.\ a cost function $W:\; \Theta\times\Theta\rightarrow \mathbb{R}^+$, if 
$\mathbb{E}_{\theta}\left(W(\theta',g(X)\right)\geq \mathbb{E}_{\theta}\left(W(\theta,g(X)\right)$, $\forall \theta',\theta\in \Theta$.
  When the cost function $W$ is equal to the MSE, the Lehmann unbiasedness degenerates to standard unbiasedness, i.e., $\mathbb{E}_{\theta}(g(X))= \theta$, $\forall \theta\in \Theta$.
} \cite{lehmann1951} is used to introduce the concept of periodic unbiasedness\footnote{
Periodic unbiasedness  \cite{RouttenbergTabrikian2011} is Lehmann unbiasedness with the cost function $\mathbb{E}\left(\mod_{2\pi}(g(X)-\theta)^2\right)$, where the $\mod_{2\pi}$ maps the squared error to $[-\pi,\pi]$. 
}, and a CR type bound on the mean square periodic error of any periodic unbiased estimator is derived. It is easy to realize that also in this case, if the set of parameters is constrained to a strictly smaller subset of $[-\pi,\pi]$ which has and extreme point, there still is no Lehmann unbiased estimator with respect to the mean square periodic error. 

\subsection{Conclusion}

In this paper, we showed that under very general conditions, biasedness of the estimator is inevitable in constrained estimation problems. 
Sufficient conditions for the non-existence of an unbiased estimator whose range is the constrained set of parameters, is that the latter would have an extreme point, and that the distribution given the extreme value of $\theta$ would be absolutely continuous w.r.t.\ another hypothesized distribution corresponding to another parameter in the constrained set. We extend the result to a case of open sets whose closure have extreme point. We also state more easily verifiable conditions which require that the constrained set be compact and the set of hypothesized distributions are continuous w.r.t.\ one another.

As mentioned in the introduction, in \cite{marzetta1993simple} the case in which there exists an unbiased estimator is considered for equality constraints. 
If the equality constraints on a non linear function $f$ define a set of solutions which is bounded, then there is no unbiased estimator. 
It should be emphasized, however, that unbiased estimators can exist in the case in which the constrained parameter set is a manifold without boundaries, e.g., a union of hyperplanes in which there are no extreme points. 

Our results have strong impact on the applicability of the results of \cite{gorman1990lower}, \cite{marzetta1993simple}, \cite{StoicaBoon1998}.


\end{document}